\documentclass[11pt]{article}
\usepackage{enumerate}
\usepackage{amssymb,a4wide,latexsym,makeidx,epsfig,fleqn}
\usepackage{amsthm}
\usepackage{amsmath}
\usepackage{enumerate}
\usepackage{graphicx}
\usepackage{float}
\usepackage{color}
\usepackage{indentfirst}
\usepackage{pdflscape}
\usepackage{siunitx}
\usepackage{diagbox}
\usepackage{rotating}
\usepackage{booktabs}
\usepackage[colorlinks, linkcolor=blue, anchorcolor=blue, citecolor=blue]{hyperref}
\usepackage[numbers,sort&compress]{natbib}

\allowdisplaybreaks[4]
\newtheorem{theorem}{Theorem}[section]

\newtheorem{lemma}[theorem]{Lemma}

\newtheorem{corollary}[theorem]{Corollary}

\begin{document}
	\textwidth 150mm \textheight 225mm
	\title{Size conditions and spectral conditions for generalized critical graphs about 
		 $k$-matching
		\thanks{Supported by the National Natural Science Foundation of China (No. 12271439)}}

	\author{{Zhenhao Zhang$^{a,b}$, Ligong Wang$^{a,b,}$\thanks{Corresponding author.}}\\
		{\small $^a$School of Mathematics and Statistics, Northwestern
			Polytechnical University,}\\ {\small  Xi'an, Shaanxi 710129,
			P.R. China.}\\
		{\small $^b$Xi'an-Budapest Joint Research Center for Combinatorics, Northwestern
			Polytechnical University,}\\
		{\small Xi'an, Shaanxi 710129,
			P.R. China. }\\
		{\small E-mail: zhangzhenhao@mail.nwpu.edu.cn, lgwangmath@163.com} }
	\date{}
	\maketitle	\begin{center}
		\begin{minipage}{120mm}
			\vskip 0.3cm
			\begin{center}
				{\small {\bf Abstract}}
			\end{center}
{\small   A $k$-barrier of a graph $G$ is the subset $S\subseteq V(G)$ that reaches the maximum value in the $k$-Berge-Tutte-formula of $G$. A graph $G$ of odd order (resp. even order) is generalized factor-critical (resp. generalized bicritical)  if $\emptyset$ is its only $k$-barrier.  	Denote by $E_G(v)$ the set of all edges incident to a vertex $v$ in $G$.			
            A $k$-matching of a graph $G$ is a function $f:E(G) \rightarrow \{0,1,\ldots,k\}$ such that $\sum_{e\in E_G(v)} f(e)\leq k$ for every vertex $v\in V(G)$.  For $1\leq d\leq k$ and $d\equiv |V(G)|~(\mathrm{mod}~2)$, if for any $ v \in V(G)$,  there exists a $k$-matching $f$ such that $\sum_{e\in E_G(v)}f(e)=k-d$  and $\sum_{e\in E_G(u)}f(e)=k \text{ for any } u\in V(G)-\{v\}$. 	Then $G$ is $k$-$d$-critical.
            In this paper, we establish tight sufficient conditions in terms of size or spectral radius respectively for a graph $G$ to be generalized factor-critical, generalized bicritical, and $k$-$d$-critical.
            Furthermore, we prove the equivalence of the existence of four factors in a graph. Thus we also give size conditions and spectral radius conditions for a graph $G-v$ to have one of the four factors for any $v\in V(G)$.

				 }
		\end{minipage}
	\end{center}
	
	\section{Introduction }
	\label{sec:ch6-introduction}
	\noindent
	
 All graphs considered in this paper are simple, connected and undirected. Let $G$ be a graph with vertex set $V(G)$ and edge set $E(G)$. The order of $G$ is the number of vertices of $G$. We use  $e(G)$ to denote the number of edges of $G$, called the size of $G$.  Let $E_G(v)=\{e \in E(G)\mid \text{$e$ is incident with  $v$ in $G$} \}$.  Let $N_G(v)$ (or simply $N(v)$) be the set of  neighbours of a vertex $v$ in $G$.
  A vertex is called an isolated vertex in a graph $G$ if it is not adjacent to any other vertex in $G$. Denote by $I(G)$ the set of isolated vertices in $G$.
Let $G$ and $H$ be two graphs. If $V(H)\subseteq V(G)$ and $E(H)\subseteq E(G)$, then $H$ is called a subgraph of $G$, denoted by $H\subseteq G$. In addition, if $H\subseteq G$ and $V(H)=V(G)$, then $H$ is a spanning subgraph of $G$, denoted by $H\preceq G$.
For $V_1 \subseteq V(G)$, let $G-V_1$ be the graph formed from $G$ by deleting all
the vertices in $V_1$ and their incident edges. $G-\{v\}$ is simply denoted by $G-v$. For two graphs $G_1$ and $G_2$, we use $G_1+G_2$ to denote the union of $G_1$ and $G_2$ with vertex set $V(G_1)\cup V(G_2)$ and edge set $E(G_1)\cup E(G_2)$. The join $G_1\vee G_2$ is the graph obtained from vertex-disjoint graphs $G_1$ and $G_2$ by adding all possible edges between $V(G_1)$ and $V(G_2)$. As usual, denote by $K_n$, $K_{a,n-a}$ and $C_n$ the complete graph, the complete bipartite graph and the cycle of order $n$, respectively. Let $\mbox{odd}(G)$ and $i(G)$ denote the number of nontrivial odd components and the number of isolated vertices of a graph $G$, respectively.	The adjacency matrix of a graph $G$ of order $n$ is defined as $A(G)=(a_{ij})_{n\times n}$, where $a_{ij} = 1$ if $v_i$ is adjacent to $v_j$, and $a_{ij} = 0$ otherwise.
The adjacency spectral radius of $G$, denoted by $\rho(G)$, is the largest eigenvalue of the nonnegative matrix $A(G)$.

A $k$-matching of a graph $G$ is a function $f\colon E(G) \rightarrow \{0,1,\ldots,k\}$ such that $\sum_{e\in E_G(v)} f(e)$ $\leq k$ for every vertex $v\in V(G)$. Let $\mu_k(G)=\mbox{max}\{\sum_{e\in E(G)} f(e)|  \text{$f$ is a $k$-matching}\}$ be the $k$-matching number of $G$.
 Similar to $k$-matching,  a fractional matching of a graph $G$ is a function $f\colon E(G) \rightarrow [0,1]$ such that $\sum_{e\in E_G(v)} f(e)$ $\leq 1$ for  every vertex $ v \in V(G)$. The fractional matching number of $G$, denoted by $\mu_f(G)$, is the maximum value of  $ \sum_{e\in E(G)} f(e)$ over all fractional matchings $f$ of $G$.
 Given subgraphs $H_1, H_2,\ldots,H_t$ of $G$,  a $\{H_1, H_2,\ldots,H_t\}$-factor of   $G$ is a spanning subgraph $F$ in which each connected component is isomorphic to one of $H_1, H_2,\ldots,H_t$.

 Liu et al. \cite{Liu1,Liu2} gave the $k$-Berge-Tutte-formula of a graph $G$:
 \[\mbox{def}_k(G)=\mathop{\mbox{max}}\limits_{S\subseteq V(G)} \begin{cases}
 	k\cdot i(G-S)-k|S|, & \text{$k$ is even};     \\
 \mbox{odd}(G-S)+k\cdot i(G-S)-k|S|, &\text{$k$ is odd}.
 \end{cases}\]
 A  $k$-barrier of a graph $G$ is the subset $S\subseteq V(G)$ that reaches the maximum value in $k$-Berge-Tutte-formula.
 Chang et al.  \cite{Chang} defined a graph to be generalized factor-critical (denoted by $\text{GFC}_k$) or to be generalized bicritical (denoted by $\text{GBC}_k$). That is a graph $G$ of odd order (resp. even order) is generalized factor-critical (resp. generalized bicritical)  if $\emptyset$ is its only $k$-barrier.

 For an odd integer $k\geq 3$, Liu et al. \cite{Liu2} established the following characterization of $\text{GFC}_k$ graphs.
\begin{lemma}(\cite{Liu2})
	If $G$ is a graph of odd order $n\geq 3$ and $k\geq 3$ is odd, then $G$ is $\text{GFC}_k$ if and only if for any $ v \in V(G)$,  there exists a $k$-matching $f$ such that
	\[ \sum_{e\in E_G(v)}f(e)=k-1 \text{ and } \sum_{e\in E_G(u)}f(e)=k \text{ for any } u\in V(G)-\{v\}.
	\]
	
\end{lemma}

  Chang et al.   \cite{Chang} extended the above results and defined a $k$-$d$-critical graph.  For an odd integer $k\geq 3$, let $1\leq d \leq k$ and $|V(G)|\equiv d~(\mathrm{mod}~2)$. If for any $ v \in V(G)$,  there exists a $k$-matching $f$ such that
  \[ \sum_{e\in E_G(v)}f(e)=k-d \text{ and } \sum_{e\in E_G(u)}f(e)=k \text{ for any } u\in V(G)-\{v\}.
  \]
   Then $G$ is $k$-$d$-critical.

 As early as 1953, Tutte   \cite{Tut} provided the definition of 2-matching and proved the equivalence of the existence of $\{K_2, \{C_{2t+1}\colon t\geq 1 \}\}$-factor and perfect 2-matching. Lu and Wang   \cite{Lu}  generalized the notions of 2-matching to $k$-matching and characterized a necessary and sufficient condition for the existence of $k$-matching in a graph.  Liu et al.  \cite{Liu1,Liu2}  proposed $k$-Berge-Tutte-formula.
  Based on this, Chang et al.   \cite{Chang} extended the above results by introducing the notions of $\text{GFC}_k$ graphs, $\text{GBC}_k$ graphs and $k$-$d$-critical graphs, and established the corresponding necessary and sufficient conditions in a graph. A $k$-matching is the usual notion of a matching when $k=1$.
   O \cite{O} provided a spectral condition for the existence of perfect $k$-matching in a graph. Zhao et al. \cite{Zhao} extended the result to the $A_\alpha$-spectral radius. Zhang and Lin \cite{Zhang1} presented a distance spectral radius condition to guarantee the existence of a perfect matching in a graph or a bipartite graph. Zhang et al. \cite{Zhang2} generalized
   and improved the previous result in \cite{Zhang1}. They investigated the existence of
   perfect matching in a bipartite graph with given minimum degree
   in terms of its distance spectral radius.
 Zhang and Fan   \cite{Zhang} gave a size condition and a spectral radius condition for a graph to have perfect $k$-matching. Niu et al.   \cite{Niu} extended the spectral radius condition to the $A_\alpha$-spectral radius condition.
 The spectral radius of a graph is intimately linked to its fundamental properties.
 Therefore, characterizing $\text{GFC}_k$ graphs, $\text{GBC}_k$ graphs and $k$-$d$-critical graphs in terms of size or spectral radius is a well-motivated problem.

The structure of this paper is as follows. In Section 2, we present some lemmas which will be used later. For odd $k\geq 3$, in Sections 3,4 and 5, we establish  tight sufficient conditions in terms of size or spectral radius respectively for a graph $G$ to be $\text{GFC}_k$, $\text{GBC}_k$ and $k$-$d$-critical.
 In Section 6, for even $k$, we establish  tight sufficient conditions in terms of size or spectral radius for a graph $G$ to be $\text{GFC}_k$ for odd order and $\text{GBC}_k$ for even order. Besides, we prove the equivalence of the existence of four factors (namely, $\{K_2,\{C_t\colon t\geq 3\}\}$-factor, $\{K_2,\{C_{2t+1}\colon t\geq 1\}\}$-factor, fractional perfect matching, perfect $k$-matching with even $k$) in a graph. Thus we also give size conditions and spectral radius conditions for a graph $G-v$ to contain one of these factors for any $v$ in $V(G)$.

	\section{Preliminaries}\label{sec:Preliminaries}
	
	In this section, we will give some useful lemmas which will be used later.
	
	\begin{lemma}\label{lem1}(\cite{Hong})
		Let $G$ be a connected graph of order $n$ and size $m$. Then $\rho(G)\leq \sqrt{2m-n+1}$.
 Equality holds if and only if $G=K_n$ or $K_{1,n-1}$.
	\end{lemma}
	
	\begin{lemma}\label{lem2}
	(\cite{Sch})
	For any graph $G$, $2\mu_f(G)$ is an integer. Moreover, there is a fractional matching $f$ for which
		\[\sum_{e\in E(G)}f(e)=\mu_f(G)  \]
		such that $f(e) \in \{0,1/2,1\}$ for every edge $e$.
	\end{lemma}
	
	\begin{lemma}\label{lem3}	(\cite{Chang})
		Let $G$ be a graph of order $n\geq 3$ and $k\geq 2$.\\
		(1) When $k$ is even, $G$ is $\text{GFC}_k$ (resp.\ $\text{GBC}_k$) if and only if $n$ is odd (resp.\ even) and
			\[
			i(G - S) \leq |S| - 1  \text{ for any } \varnothing \neq S \subseteq V(G).
			\]
		(2) When $k$ is odd, $G$ is $\text{GFC}_k$ if and only if $n$ is odd and
			\[
			\mbox{odd}(G - S) + k \cdot i(G - S) \leq k|S| - 1  \text{ for any } \varnothing \neq S \subseteq V(G).
			\]
		(3) When $k$ is odd, $G$ is $\text{GBC}_k$ if and only if $n$ is even and
			\[
			\mbox{odd}(G - S) + k \cdot i(G - S) \leq k|S| - 2  \text{ for any } \varnothing \neq S \subseteq V(G).
			\]

	\end{lemma}
	
	\begin{lemma}\label{lem6}	(\cite{Chang})
		Let $G$ be a  graph of order $n\geq 3$, $k\geq 3$ is odd and let $1\leq d \leq k$ with  $n\equiv d~(\mathrm{mod}~2)$. $G$ is $k$-$d$-critical if and only if
		\[\mbox{odd}(G-S)+k\cdot i(G-S) \leq
		k|S|-d \text{ for any } \varnothing \neq S \subseteq V(G).
		\]

	\end{lemma}

	\begin{lemma}\label{lem4}	(\cite{Zheng})
		Let $n = s + \sum_{i=1}^t n_i $. If $n_1 \geq n_2 \geq \cdots \geq n_t \geq p \geq 1 $ and $ n_1 < n - s - p(t-1) $, then
		\[
		e\bigl(K_s \vee (K_{n_1} + \cdots + K_{n_t})\bigr) < e\bigl(K_s \vee (K_{n-s-p(t-1)} + (t-1)K_p)\bigr).
		\]

	\end{lemma}
	
		\begin{lemma}\label{lem5}	(\cite{Ber})
		Let $A = (a_{ij} )$ and $B = (b_{ij} )$ be two nonnegative square matrices of
		order $n$. Let $\rho(A)$ be the largest eigenvalue of $A$. If $A \geq  B$, then $\rho(A) \geq \rho(B)$. Furthermore, if $B$ is irreducible, $A \leq B$ and $A\not = B$, then $\rho(A)<\rho(B)$.
	\end{lemma}
	
			\begin{lemma}\label{lem7}	(\cite{Chang})
		 Let $G$ be a connected graph of order $n\geq 3$ and let $k$ be even. Then $G$ is $\text{GFC}_k$ (resp.\ $\text{GBC}_k$) if and only if $n$ is odd (resp.\ even)
		and $G-v$ has a perfect $k$-matching for any $v\in V(G)$.
	\end{lemma}

	\section{ Size Condition and Spectral Radius Condition for a Graph to be $\text{GFC}_k$ with odd $k\geq 3$.}
	
	For a graph $G$ of odd order $n$, we establish  tight sufficient conditions in terms of size or spectral radius for $G$ to be $\text{GFC}_k$, where $k\geq 3$ is an odd integer. Theorem    \ref{thm1} gives the size condition for $G$ to be $\text{GFC}_k$. Theorem    \ref{thm2} gives the spectral radius condition for $G$ to be $\text{GFC}_k$.
	
	\begin{theorem}\label{thm1}
		For odd $n\geq 3$, and odd $k\geq 3$, let $G$ be a graph of order $n$.\\
	(1)	If $n\geq 7$ or $n=3$, and $e(G)\geq \binom{n-1}{2}+1$, then $G$ is $\text{GFC}_k$ unless $ G = K_1 \vee (K_{n-2}+K_1)$.\\
	(2) If $n=5$ and $e(G)\geq 7$, then $G$ is $\text{GFC}_k$ unless $ G = K_1 \vee (K_3+K_1)$ or $G =K_2 \vee 3K_1$.
		
	\end{theorem}

	\begin{proof}

	For odd $n\geq 3$ and odd $k\geq 3$, let $G$ be a graph of order $n$ that satisfies the conditions of Theorem   \ref{thm1}.
	Assume to the contrary that $G$ is not $\text{GFC}_k$.
	 By Lemma    \ref{lem3} (2), there exists $\varnothing \neq S \subseteq V(G)$ such that \begin{equation}\label{1}
	 \mbox{odd}(G-S)+k\cdot i(G-S)\geq k|S|. \tag{1}
	 \end{equation}
	 We now proceed to discuss the following two cases.

		{\textbf{Case 1.}} $\mbox{odd}(G-S)>0$.

	In this case, we may assume, without loss of generality, that $G-S$ contains no even components. Otherwise, we construct a new graph $G'$ by adding all possible edges between one odd component and all the even components. Then $\mbox{odd}(G'-S)=\mbox{odd}(G-S)$ and $i(G'-S)=i(G-S)$. We have $\mbox{odd}(G'-S)+k\cdot i(G'-S)\geq k|S|$, but $e(G')>e(G)$. Thus we obtain a graph $G'$ without even components. If $G$ originally has no even components, set $G'=G$.
	
	Let the odd components of $G'-S$ be $G'_{1}$, $G'_{2}$, \ldots, $G'_{p}$. Denote $|S|=s, |V(G'_1)|=n_1,\ldots, |V(G'_p)|=n_p$. Then $G'$ is a spanning subgraph of $K_s \vee (K_{n_1}+K_{n_2}+\cdots+K_{n_p})$.
	By Lemma \ref{lem4}, $e(K_s \vee (K_{n_1}+K_{n_2}+\cdots+K_{n_p}))\leq e(K_s \vee (K_{n_1+n_2+\ldots+n_p-p+1}+(p-1)K_1))$.
	Let $G''=K_s \vee (K_{n_1+n_2+\ldots+n_p-p+1}+(p-1)K_1)$.
	Hence $\mbox{odd}(G''-S)=1$, by Inequality (\ref{1}) we have $1+k\cdot i(G''-S)\geq ks$. Then $i(G''-S)\geq s$. Then $G''$ must be a spanning subgraph of $K_s \vee (K_{n-2s}+sK_1)$. Since $K_{n-2s}$ is a nontrivial odd component,  $n-2s\geq 3$. Hence $1\leq s\leq \frac{n-3}{2}$.

    Let $f(s)=e(K_s \vee (K_{n-2s}+sK_1))=\frac{3s^2}{2}+(\frac{1}{2}-n)s$+$\frac{n^2-n}{2}$. Since $f(s)$ is  a quadratic function in the variable $s$ that opens upward and its axis of symmetry is at $s=\frac{2n-1}{6}$ and $(\frac{2n-1}{6}-1)-(\frac{n-3}{2}-\frac{2n-1}{6})=\frac{n+1}{6}>0$.  Thus $f(1)>f(\frac{n-3}{2})$. Therefore, $e(G)\leq e(G') \leq e(G'')\leq e(K_1 \vee (K_{n-2}+K_1))=\binom{n-1}{2}+1$  with equality throughout if and only if $G=K_1 \vee (K_{n-2}+K_1)$.

	{\textbf{Case 2.}} $\mbox{odd}(G-S)=0$.
	
	In this case, there are no nontrivial odd components in $G-S$. We may assume, without loss of generality, that there are no even components in $G-S$. Otherwise, we can construct a new graph $G'''$ by connecting all vertices of the even components to an isolated vertex.
	Then $\mbox{odd}(G'''-S)=1$ and $i(G'''-S)=i(G-S)-1$. Since $i(G-S)+|S|$ and $n$ have the same parity,  $i(G-S)+|S|$ is odd. By Inequality (\ref{1}) we have  $i(G-S)\geq |S|$. Then $i(G-S)\geq |S|+1$. In $G'''$, $\mbox{odd}(G'''-S)+k\cdot i(G'''-S)=1+k\cdot(i(G-S)-1)\geq 1+k|S|>k|S|$. Then $G'''$ is still not $\text{GFC}_k$ with $\mbox{odd}(G'''-S)=1$. It can be reduced to Case 1.
		
	We now consider that $G-S$ has no even components and no nontrivial odd components. Let $|S|=s$ and $i(G-S)=i$, then $G\preceq K_s \vee iK_1$. Since $n$ is odd, $e(K_s \vee iK_1)\leq e(K_\frac{n-1}{2}\vee \frac{n+1}{2}K_1)$. Therefore, $e(G) \leq e(K_\frac{n-1}{2}\vee \frac{n+1}{2}K_1)$ with equality if and only if $G=K_\frac{n-1}{2}\vee \frac{n+1}{2}K_1$.

	Now we compare the number of edges in the two cases: $K_1 \vee (K_{n-2}+K_1)$ and $K_\frac{n-1}{2}\vee \frac{n+1}{2}K_1$. Denote $g(n)=e(K_1 \vee (K_{n-2}+K_1))-e(K_\frac{n-1}{2}\vee \frac{n+1}{2}K_1)$. Then $g(n)=\frac{(n-3)(n-5)}{8} > 0$ for $n\geq 7$. For $n=3$, $K_1 \vee (K_1+K_1)$ and $K_1\vee 2K_1$ are isomorphic. For $n=5$, $K_1\vee(K_3+K_1)$ and $K_2\vee 3K_1$ are non-isomorphic, but they have the same size, $e(K_1 \vee (K_3+K_1))=e(K_2\vee3K_1)=7$. 
	
	In conclusion, if  $G$ is not $\text{GFC}_k$, either $G=K_1 \vee (K_{n-2}+K_1)$ or $e(G)<e(K_1 \vee (K_{n-2}+K_1))=\binom{n-1}{2}+1$ for $n\geq 7$ or $n=3$;  If  $G$ is not $\text{GFC}_k$, then $G=K_1 \vee (K_3+K_1)$, $e(G)<7$ or $G =K_2\vee 3K_1$ for $n=5$. It leads to a contradiction with Theorem   \ref{thm1}. This completes the proof.
	
\end{proof}

	\begin{theorem} \label{thm2}
	For odd $n\geq 3$ and odd $k\geq 3$, let $G$ be a graph of order $n$.
	If $\rho(G)\geq \rho(K_1 \vee (K_{n-2}+K_1))$, then $G$ is $\text{GFC}_k$ unless $ G = K_1 \vee (K_{n-2}+K_1)$.
	
\end{theorem}	

\begin{proof}

	For odd $n\geq 3$ and odd $k\geq 3$, let $G$ be a graph of order $n$ that satisfies the conditions of  Theorem   \ref{thm2}. Assume to the contrary that $G$ is not $\text{GFC}_k$.
	
	By Theorem   \ref{thm1}, if $G\not =K_1 \vee (K_{n-2}+K_1)$, then $e(G)<\binom{n-1}{2}+1=\frac{n^2-3n+4}{2}$ for $n\geq 7$ or $n=3$. Thus $e(G)\leq \frac{n^2-3n+2}{2}$. According to Lemma   \ref{lem1}, $\rho(G)\leq \sqrt{2\frac{n^2-3n+2}{2}-n+1}\leq \sqrt{n^2-4n+3}<n-2$. However, $\rho(G)\geq \rho(K_1 \vee (K_{n-2}+K_1))\geq \rho(K_{n-1})=n-2$, which leads to a contradiction. When $n=5$, if $G\not =K_1 \vee (K_3+K_1)$ and $e(G)\geq \binom{4}{2}+1=7$, then $G=K_2\vee 3K_1$. But $\rho(K_2\vee 3K_1)=3<\rho(K_1 \vee (K_3+K_1))=3.09$,  which also leads to a contradiction.
	 This completes the proof.
\end{proof}

\section{ Size Condition and Spectral Radius Condition  for a Graph to be $\text{GBC}_k$ with odd $k\geq 3$.}
	
	For a graph $G$ of even order $n$, we establish  tight sufficient conditions in terms of size or spectral radius for $G$ to be $\text{GBC}_k$, where $k\geq 3$ is an odd integer. Theorem    \ref{thm3} gives the size condition for $G$ to be $\text{GBC}_k$. Theorem    \ref{thm4} gives the spectral radius condition for $G$ to be $\text{GBC}_k$.
	
		\begin{theorem}\label{thm3}
		For even $n\geq 4$ and odd $k\geq 3$, let $G$ be a graph of order $n$.\\
		(1)	If $n \geq 10$ and $e(G)\geq \binom{n-1}{2}+1$, then $G$ is $\text{GBC}_k$ unless $ G = K_1 \vee (K_{n-2}+K_1)$.\\
		(2)	If $4\leq n \leq 6$ and $e(G)\geq \frac{3n^2-2n}{8}$, then $G$ is $\text{GBC}_k$ unless $ G = K_\frac{n}{2} \vee \frac{n}{2}K_1$.\\
		(3) 	If $n =8$ and $e(G)\geq 22$, then $G$ is $\text{GBC}_k$ unless $ G = K_4 \vee 4K_1$ or $ G = K_1 \vee (K_6+K_1)$.
	\end{theorem}

	\begin{proof}

For even $n\geq 4$ and odd $k\geq 3$, let $G$ be a graph of order $n$ that satisfies the conditions of Theorem   \ref{thm3}. Assume to the contrary that $G$ is not $\text{GBC}_k$. By Lemma    \ref{lem3} (3), there exists $\varnothing \neq S \subseteq V(G)$ such that
\begin{equation}\label{2}
	\mbox{odd}(G-S)+k\cdot i(G-S)\geq k|S|-1. \tag{2}
\end{equation}
We now proceed to discuss the following two cases.
	
	{\textbf{Case 1.}} $\mbox{odd}(G-S)>0$.

	In this case, we may assume, without loss of generality, that $G-S$ contains no even components. Otherwise, we construct a new graph $G'$ by adding all possible edges between one odd component and all the even components.
Then $\mbox{odd}(G'-S)=\mbox{odd}(G-S)$ and $i(G'-S)=i(G-S)$. We have $\mbox{odd}(G'-S)+k\cdot i(G'-S)\geq k|S|-1$, but $e(G')>e(G)$. Thus we obtain a graph $G'$ without even components. If $G$ originally has no even components, set $G'=G$.
	
		Let the odd components of $G'-S$ be $G'_{1}$, $G'_{2}$, $\ldots$, $G'_{p}$. Denote $|S|=s, |V(G'_1)|=n_1, \ldots , |V(G'_p)|=n_p$. Then $G'$ is a spanning subgraph of $K_s \vee (K_{n_1}+K_{n_2}+ \cdots +K_{n_p})$. By Lemma \ref{lem4}, $e(K_s \vee (K_{n_1}+K_{n_2}+\cdots+K_{n_p}))\leq e(K_s \vee (K_{n_1+n_2+\ldots+n_p-p+1}+(p-1)K_1))$. Let $G''=K_s \vee (K_{n_1+n_2+\ldots+n_p-p+1}+(p-1)K_1)$. Hence $\mbox{odd}(G''-S)=1$, by Inequality (\ref{2}) we have $1+k\cdot i(G''-S)\geq ks-1$. Then $i(G''-S)\geq s$ as $k\geq 3$. $\mbox{odd}(G''-S)+i+s$ and $n$ have the same parity, then $i+s$ is odd. Then $i(G''-S)\geq s+1$ and $G''$ is a spanning subgraph of $K_s \vee (K_{n-2s-1}+(s+1)K_1)$. Since $K_{n-2s}$ is a nontrivial odd component, $n-2s-1\geq 3$. Hence $1\leq s\leq \frac{n-4}{2}$.

	Let $f(s)=e(K_s \vee (K_{n-2s-1}+(s+1)K_1))=\frac{3s^2}{2}+(\frac{5}{2}-n)s$+$\frac{n^2-3n}{2}$. Since $f(s)$ is  a quadratic function in the variable $s$ that opens upward and its axis of symmetry is at $s=\frac{2n-5}{6}$ and $(\frac{2n-5}{6}-1)-(\frac{n-4}{2}-\frac{2n-5}{6})=\frac{n-4}{6}\geq 0$, $f(1)\geq f(\frac{n-4}{2})$. Therefore, $e(G)\leq e(G')\leq e(G'')\leq e(K_1 \vee (K_{n-3}+2K_1))$ with equality throughout if and only if $G=K_1 \vee (K_{n-3}+2K_1)$.
	
		{\textbf{Case 2.}} $\mbox{odd}(G-S)=0$.
		
		In this case, there are no nontrivial odd components in $G-S$. Let $|S|=s, i(G-S)=i$. Since $\mbox{odd}(G-S)=0$, $ki\geq ks-1$.
		
		{\textbf{Subcase 2.1.}} $i>s$.
		
		We may assume, without loss of generality, that there are no even components in $G-S$. Otherwise, we can construct a new graph $G'''$ by connecting all vertices of the even components to an isolated vertex. Then $\mbox{odd}(G'''-S)=1$ and $i(G'''-S)=i(G-S)-1$. In $G'''$, $\mbox{odd}(G'''-S)+k\cdot i(G'''-S)=1+k\cdot(i(G-S)-1)\geq 1+k|S|>k|S|$. Then $G'''$ is still not $\text{GBC}_k$ with $\mbox{odd}(G'''-S)=1$. It can be reduced to Case 1.
        Then there are  no nontrivial odd components and even components in $G-S$, $G$ is a spanning subgraph of $K_s\vee iK_1$. Since $i>s$, $e(G)\leq e(K_s\vee iK_1)\leq e(K_{\frac{n}{2}-1}\vee (\frac{n}{2}+1)K_1)$ with equality if and only if $G=K_{\frac{n}{2}-1}\vee (\frac{n}{2}+1)K_1$.

		{\textbf{Subcase 2.2.}} $i=s$.
		
		In this subcase, $G\preceq K_s\vee (K_{n-2s}+sK_1)$ with $1\leq s \leq \frac{n}{2}$. $e(G)=\binom{n-s}{2}+s^2=\frac{3s^2}{2}+(\frac{1}{2}-n)s$+$\frac{n^2-n}{2}$. It is   a quadratic function in the variable $s$ that opens upward and its axis of symmetry is at $s=\frac{2n-1}{6}$.  $(\frac{2n-1}{6}-1)-(\frac{n}{2}-\frac{2n-1}{6})=\frac{n-8}{6}$. So $e(K_1 \vee (K_{n-2}+K_1))> e(K_\frac{n}{2}\vee \frac{n}{2}K_1)$ for $n\geq 10$; $e(K_1 \vee (K_6+K_1))= e(K_4\vee 4K_1)$ for $n=8$; $e(K_1 \vee (K_{n-2}+K_1))<e(K_\frac{n}{2}\vee \frac{n}{2}K_1)$ for $n\leq 6$.

	\vspace{\baselineskip}
	
	Now we  compare the number of edges in the two cases. We observe that $K_1 \vee (K_{n-3}+2K_1)$ is a subgraph of $K_1 \vee (K_{n-2}+K_1)$, $K_{\frac{n}{2}-1}\vee (\frac{n}{2}+1)K_1$ is a subgraph of $K_\frac{n}{2}\vee \frac{n}{2}K_1$. Therefore, we only need to compare the two graphs in Subcase 2.2, and the relevant results have already been computed in Subcase 2.2.
	
	If  $G$ is not $\text{GBC}_k$, either $G=K_1 \vee (K_{n-2}+K_1)$ or $e(G)<e(K_1 \vee (K_{n-2}+K_1))=\binom{n-1}{2}+1$ with $n \geq 10$; if  $G$ is not $\text{GBC}_k$, then $G=K_4\vee 4K_1$, $G=K_1 \vee (K_6+K_1)$ or $e(G)<e( K_4\vee 4K_1 )=22$ with $n =8$; if  $G$ is not $\text{GBC}_k$, either $G=K_\frac{n}{2}\vee \frac{n}{2}K_1$ or $e(G)<e( K_\frac{n}{2}\vee \frac{n}{2}K_1 )=\frac{3n^2-2n}{8}$ with $n \leq 6$. In any case, the above leads to a contradiction with  Theorem   \ref{thm3}. This completes the proof.
		\end{proof}
	
		\begin{theorem} \label{thm4}
		For even $n\geq 4$ and odd $k\geq 3$, let $G$ be a graph of order $n$.\\
		(1)	If $n \geq 8$ and $\rho(G)\geq \rho(K_1 \vee (K_{n-2}+K_1))$, then $G$ is $\text{GBC}_k$ unless $ G = K_1 \vee (K_{n-2}+K_1)$.\\
		(2)  If $4\leq n \leq 6$ and  $\rho(G)\geq \rho(K_\frac{n}{2} \vee \frac{n}{2}K_1)$, then $G$ is $\text{GBC}_k$ unless $ G = K_\frac{n}{2} \vee \frac{n}{2}K_1$.
	\end{theorem}	
	
\begin{proof}

	\begin{table}[htbp]
	\centering
	\caption{Spectral radii of $G_s=K_s \vee (K_{n-2s}+sK_1)$ and $K_\frac{n}{2}\vee \frac{n}{2}K_1$}
	\label{example}
	\begin{tabular}{|c|*{4}{c|}}
		\hline
		$\rho(K_\frac{n}{2}\vee \frac{n}{2}K_1)$ & \diagbox{$n$}{$G_s$} & $\rho(G_1)$ & $\rho(G_2)$ & $\rho(G_3)$ \\
		\hline
		2.56 & $n=4$ & 2.17 & & \\
		\hline
		4.16 & $n=6$ & 4.05 & 3.62 & \\
		\hline
		5.77 & $n=8$ & 6.02 & 5.27 & 5.18 \\
		\hline
	\end{tabular}
	\label{tab:spectral_data}
\end{table}
For even $n\geq 4$ and odd $k\geq 3$, let $G$ be a graph of order $n$ that satisfies the conditions of Theorem   \ref{thm4}. Assume to the contrary that $G$ is not $\text{GBC}_k$.

From Theorem  \ref{thm3}, when $n\geq 10$, if $G\not =K_1 \vee (K_{n-2}+K_1)$, then $e(G)<\binom{n-1}{2}+1=\frac{n^2-3n+4}{2}$. Thus $e(G)\leq \frac{n^2-3n+2}{2}$. According to Lemma   \ref{lem1}, $\rho(G)\leq \sqrt{2 \frac{n^2-3n+2}{2}-n+1}\leq \sqrt{n^2-4n+3}<n-2$. However, $\rho(G)\geq \rho(K_1 \vee (K_{n-2}+K_1))\geq \rho(K_{n-1})=n-2$, which leads to a contradiction.

For $4 \leq n \leq 8$, we can compute the spectral radii directly. Through an analysis similar to the proof of Theorem \ref{thm3}, we can conclude that the graph is one of the following graphs: $K_s\vee (K_{n-2s}+sK_1)$ with $1\leq s < \frac{n}{2}$, $K_\frac{n}{2}\vee \frac{n}{2}K_1$. Denote $G_s=K_s\vee (K_{n-2s}+sK_1)$ and calculation results are presented in Table \ref{example}. When $n=8$, the graph that is not $\text{GBC}_k$ with the maximum spectral radius is $K_1 \vee (K_6+K_1)$; when $n=4,6$, it is $K_\frac{n}{2}\vee \frac{n}{2}K_1$.
 For a graph of order 8, if $\rho(G)\geq \rho(K_1 \vee (K_6+K_1))$, either $G$ is $\text{GBC}_k$ or $G =K_1 \vee (K_6+K_1)$. This contradicts our assumption. For a graph of order 4 or 6, if $\rho(G)\geq \rho(K_\frac{n}{2}\vee \frac{n}{2}K_1)$, either $G$ is $\text{GBC}_k$ or $G =K_\frac{n}{2}\vee \frac{n}{2}K_1$. This also contradicts our assumption. This completes the proof.

\end{proof}

	\section{ Size Condition and Spectral Radius Condition  for a Graph to be  $k$-$d$-critical with odd $k\geq 3$.}
	
		Let $G$ be a  graph of order $n\geq 3$. $k\geq 3$ is odd and let $1\leq d < k$ with  $n\equiv d~(\mathrm{mod}~2)$. We establish tight sufficient conditions in terms of size or spectral radius for $G$ to be $k$-$d$-critical. Theorem    \ref{thm5} gives the size condition for $G$ to be $k$-$d$-critical. Theorem    \ref{thm6} gives the spectral radius condition for $G$ to be $k$-$d$-critical.
	
	\noindent \textbf{Remark.}
	If $d=k$ and $n$ is even, the complete graph $K_n$ is not $k$-$d$-critical. Therefore, there does not exist an upper bound on the number of edges such that any graph with more edges is $k$-$d$-critical. In the following discussion, we only consider the case where $1\leq d< k$.
	
		\begin{theorem}\label{thm5}
		Let $G$ be a  graph of order $n\geq 3$. $k\geq 3$ is odd and let $1\leq d < k$ with  $n\equiv d~(\mathrm{mod}~2)$.\\
	(1)	If $n \geq 9$ or $n=3, 7$, and $e(G)\geq \binom{n-1}{2}+1$, then $G$ is $k$-$d$-critical unless $ G = K_1 \vee (K_{n-2}+K_1)$.\\
	(2)	If $n = 4, 6$ and $e(G)\geq \frac{3n^2-2n}{8}$, then $G$ is $k$-$d$-critical unless $ G = K_\frac{n}{2} \vee \frac{n}{2}K_1$.\\
	(3) If $n =5$ and $e(G)\geq 7$, then $G$ is $k$-$d$-critical unless $ G = K_2 \vee 3K_1$ or $ G = K_1 \vee (K_3+K_1)$.\\
	(4) 	If $n =8$ and $e(G)\geq 22$, then $G$ is $k$-$d$-critical unless $ G = K_4 \vee 4K_1$ or $ G = K_1 \vee (K_6+K_1)$.
	\end{theorem}
	
\begin{proof}

	For $n\geq 4$ and odd $k\geq 3$, let $1\leq d < k$ with  $n\equiv d~(\mathrm{mod}~2)$. Let $G$ be a graph of order $n$ that satisfies the conditions of Theorem   \ref{thm5}. Assume to the contrary that $G$ is not $k$-$d$-critical. Choose a graph $G$ whose size is as large as possible. By Lemma \ref{lem6}, there exists $\varnothing \neq S \subseteq V(G)$ such that
	\begin{equation}\label{3}
		\mbox{odd}(G-S)+k\cdot i(G-S)\geq k|S|-d+1. \tag{3}
	\end{equation}
 We now proceed to discuss  the following two cases.

	{\textbf{Case 1.}} $\mbox{odd}(G-S)>0$.
	
	In this case, we may assume, without loss of generality, that $G-S$ contains no even components. Otherwise, we construct a new graph $G'$ by adding all possible edges between one odd component and all the even components. Then $\mbox{odd}(G'-S)=\mbox{odd}(G-S)$ and $i(G'-S)=i(G-S)$. We have $\mbox{odd}(G'-S)+k\cdot i(G'-S)\geq k|S|-d+1$, but $e(G')>e(G)$. Thus we obtain a graph $G'$ without even components. If $G$ originally has no even components, set $G'=G$.
	
	Let the odd components of $G'-S$ be $G'_{1}$, $G'_{2}$, \ldots, $G'_{p}$. Denote $|S|=s, |V(G'_1)|=n_1,\ldots, |V(G'_p)|=n_p$. Then $G'$ is a spanning subgraph of $K_s \vee (K_{n_1}+K_{n_2}+\cdots+K_{n_p})$. By Lemma \ref{lem4}, $e(K_s \vee (K_{n_1}+K_{n_2}+\cdots+K_{n_p}))\leq e(K_s \vee (K_{n_1+n_2+\ldots+n_p-p+1}+(p-1)K_1))$.  Let $G''=K_s \vee (K_{n_1+n_2+\ldots+n_p-p+1}+(p-1)K_1)$. Hence, $\mbox{odd}(G''-S)=1$, by Inequality (\ref{3}) we have  $1+k\cdot i(G''-S)\geq ks-d+1$. Then $i(G''-S)\geq s-\frac{d}{k}$ which further implies $i(G''-S)\geq s$. $1+i+s$ and $n$ have the same parity. (1) $n$ is odd, then $i+s$ is even.  Hence $i(G''-S)\geq s$ and $G''\leq K_s \vee (K_{n-2s}+sK_1)$; (2) $n$ is even, then $i+s$ is odd and $i\geq s+1$.  Hence $G''\leq K_s \vee (K_{n-2s-1}+(s+1)K_1)$.
	Following a calculation similar to that in Section 3 and Section 4, we find that $e(G)\leq e(G')\leq e(G'')\leq e( K_1 \vee (K_{n-2}+K_1))$ with odd $n$ and $e(G)\leq e(G') \leq e(G'')\leq e( K_1 \vee (K_{n-3}+2K_1))$ with even $n$.
	
	{\textbf{Case 2.}} $\mbox{odd}(G-S)=0$.
	
	In this case, there are  no nontrivial odd components in $G-S$. Let $|S|=s, i(G-S)=i$. Since $\mbox{odd}(G-S)=0$, $ki\geq ks-d+1$.
	
	{\textbf{Subcase 2.1.}} $i>s$.
	
	We may assume, without loss of generality, that there are no even components in $G-S$. Otherwise, we can construct a new graph $G'''$ by connecting all vertices of the even components to an isolated vertex. Then $\mbox{odd}(G'''-S)=1$ and $i(G'''-S)=i(G-S)-1$. In $G'''$, $\mbox{odd}(G'''-S)+k\cdot i(G'''-S)=1+k\cdot(i(G-S)-1)\geq 1+k|S|>1+k|S|-d$. Then $G'''$ is still not $k$-$d$-critical with $\mbox{odd}(G'''-S)=1$. It can be reduced to Case 1. There are  no nontrivial odd components and even components in $G-S$, $G\preceq K_s\vee iK_1$. Since $i>s$, $e(G)\leq e(K_{\frac{n-1}{2}}\vee (\frac{n+1}{2})K_1)$  with odd $n$; $e(G)\leq e(K_{\frac{n}{2}-1}\vee (\frac{n}{2}+1)K_1)$ with even $n$.
	
	{\textbf{Subcase 2.2.}} $i=s$.
	
    In this subcase, $G\preceq K_s\vee (K_{n-2s}+sK_1)$ with $1\leq s \leq \frac{n}{2}$.
	Following the calculation in Section 4, we have $e(G)\leq e(K_\frac{n}{2}\vee \frac{n}{2}K_1)$ with $n\leq 6$; $e(G)\leq e(K_1 \vee (K_{n-2}+K_1))$ with $n\geq 10$; $e(G)\leq e(K_4\vee 4K_1)=e(K_1 \vee (K_6+K_1))$ with $n=8$.

	\vspace{\baselineskip}
	
	Now we compare the number of edges in the cases discussed above. We observe that $K_1 \vee (K_{n-3}+2K_1)$ is a subgraph of $K_1 \vee (K_{n-2}+K_1)$, $K_{\frac{n}{2}-1}\vee (\frac{n}{2}+1)K_1$ is a subgraph of $K_\frac{n}{2}\vee \frac{n}{2}K_1$.
	 Therefore, we only need to compare the three graphs: $K
	 _1 \vee (K_{n-2}+K_1)$, $K_\frac{n}{2}\vee \frac{n}{2}K_1$ and $G=K_{\frac{n-1}{2}}\vee \frac{n+1}{2}K_1$.  $e(K_1 \vee (K_{n-2}+K_1))-e(K_\frac{n-1}{2}\vee \frac{n+1}{2}K_1)=\frac{(n-3)(n-5)}{8} > 0$ for all odd $n\geq 7$. $K_1 \vee (K_1+K_1)=K_1\vee 2K_1$ with $n=3$ and $e(K_1 \vee (K_3+K_1))=e(K_2\vee 3K_1)$ with $n=5$.
	
	  \[e(K_1 \vee (K_{n-2}+K_1)) - e(K_\frac{n}{2}\vee \tfrac{n}{2}K_1)=\frac{(n-2)(n-8)}{8} \begin{cases}
	 	>0, & \text{for even $n\geq 10$};     \\
	 	=0, & \text{for $n=8$};\\
	 	< 0, &\text{for even $n\leq 6$}.
	 	
	 \end{cases}\]

    In conclusion, if  $G$ is not $k$-$d$-critical, either $G=K_1 \vee (K_{n-2}+K_1)$ or $e(G)<e(K_1 \vee (K_{n-2}+K_1))=\binom{n-1}{2}+1$ with $n \geq 9$ or $n=3, 7$; if  $G$ is not $k$-$d$-critical, then $G=K_1 \vee (K_3+K_1)$, $G=K_2\vee 3K_1$ or $e(G)<e(K_1 \vee (K_{n-2}+K_1))=7$ with $n=5$; if  $G$ is not $k$-$d$-critical, either $G=K_\frac{n}{2}\vee \frac{n}{2}K_1$ or $e(G)<e( K_\frac{n}{2}\vee \frac{n}{2}K_1 )=\frac{3n^2-2n}{8}$ with $n =4,6$; if  $G$ is not $k$-$d$-critical, then $G=K_1 \vee (K_6+K_1)$, $G=K_4 \vee 4K_1$ or $e(G)<22$ with $n =8$. In any case, the above leads to a contradiction with Theorem   \ref{thm5}. This completes the proof.
	\end{proof}

		\begin{theorem} \label{thm6}
		For $n\geq 3$ and odd $k\geq 3$, let $G$ be a graph of order $n$ and let $1\leq d < k$ with  $n\equiv d~(\mathrm{mod}~2)$.\\
		(1) If $n \geq 7$ or $n=3, 5$, and $\rho(G)\geq \rho(K_1 \vee (K_{n-2}+K_1))$, then $G$ is $k$-$d$-critical unless $ G = K_1 \vee (K_{n-2}+K_1)$.\\
		(2) If $ n =4, 6$ and  $\rho(G)\geq \rho(K_\frac{n}{2} \vee \frac{n}{2}K_1)$, then $G$ is $k$-$d$-critical unless $ G = K_\frac{n}{2} \vee \frac{n}{2}K_1$.
	\end{theorem}

\begin{proof}
	For $n\geq 4$ and odd $k\geq 3$, let $1\leq d < k$ with  $n\equiv d~(\mathrm{mod}~2)$. Let $G$ be a graph of order $n$ that satisfies the conditions of Theorem   \ref{thm6}. Assume to the contrary that $G$ is not $k$-$d$-critical.
	
From Theorem  \ref{thm5},  when $n\geq 9$ or $n=3,7$, if $G\not=K_1 \vee (K_{n-2}+K_1)$, then $e(G)<\binom{n-1}{2}+1=\frac{n^2-3n+4}{2}$. Thus $e(G)\leq \frac{n^2-3n+2}{2}$. According to Lemma   \ref{lem1}, $\rho(G)\leq \sqrt{2 \frac{n^2-3n+2}{2}-n+1}\leq \sqrt{n^2-4n+3}<n-2$. However, $\rho(G)\geq \rho(K_1 \vee (K_{n-2}+K_1))\geq \rho(K_{n-1})=n-2$, which leads to a contradiction.  When $n=5$, if $G\not =K_1 \vee (K_3+K_1)$ and $e(G)\geq \binom{4}{2}+1=7$, then $G=K_2\vee 3K_1$. But $\rho(K_2\vee 3K_1)=3<\rho(K_1 \vee (K_3+K_1))=3.09$,  which also leads to a contradiction.

When $n=4,6,8$, Table \ref{tab:spectral_data} shows that the graph of order $8$ that is not $k$-$d$-critical with the maximum spectral radius is $K_1 \vee (K_6+K_1)$; if $n=4,6$, it is $K_\frac{n}{2}\vee \frac{n}{2}K_1$. For a graph of order 8, if $\rho(G)\geq \rho(K_1 \vee (K_6+K_1))$, either $G$ is $k$-$d$-critical or $G =K_1 \vee (K_6+K_1)$. This contradicts our assumption. For a graph of order 4 or 6, if $\rho(G)\geq \rho(K_\frac{n}{2}\vee \frac{n}{2}K_1)$, either $G$ is $k$-$d$-critical or $G =K_\frac{n}{2}\vee \frac{n}{2}K_1$. This also contradicts our assumption. This completes the proof.
\end{proof}

	\section{ Size Condition and Spectral Radius Condition for a Graph to be $\text{GFC}_k$ or $\text{GBC}_k$ with even $k$.}
	
	For a graph $G$ of order $n$, we establish  tight sufficient conditions in terms of size or spectral radius for $G$ to be $\text{GFC}_k$ or $\text{GBC}_k$, where $k\geq 2$ is an even integer. Theorem   \ref{thm7} gives the  size condition for $G$ to be $\text{GFC}_k$ or $\text{GBC}_k$. Theorem    \ref{thm8} gives the spectral radius condition for $G$ to be $\text{GFC}_k$ or $\text{GBC}_k$. Besides, we prove the equivalence of the existence of some factors in Theorem \ref{thm9}. Thus we establish tight sufficient conditions in terms of size or spectral radius under which $G-v$ contains a $\{K_2,\{C_t: t\geq 3\}\}$-factor for any $v$ in $V(G)$; we also establish tight sufficient conditions in terms of size or spectral radius under which $G-v$ contains a fractional perfect matching  for any $v$ in $V(G)$.

	\begin{theorem}\label{thm7}
		For $n\geq 3$ and even $k$, let $G$ be a graph of order $n$.\\
		(1)	If $n \geq 9$ or $n=3,7$, and $e(G)\geq \binom{n-1}{2}+1$, then $G$ is  $\text{GFC}_k$ for odd $n$ and $\text{GBC}_k$ for even $n$ unless $ G = K_1 \vee (K_{n-2}+K_1)$.\\
		(2)	If $n=4,6$ and $e(G)\geq \frac{3n^2-2n}{8}$, then $G$ is $\text{GBC}_k$ unless $ G = K_\frac{n}{2} \vee \frac{n}{2}K_1$.\\
		(3) 	If $n =5$ and $e(G)\geq 7$, then $G$ is $\text{GFC}_k$ unless $ G = K_2 \vee 3K_1$ or $ G = K_1 \vee (K_3+K_1)$.	\\
		(4) 	If $n =8$ and $e(G)\geq 22$, then $G$ is $\text{GBC}_k$ unless $ G = K_4 \vee 4K_1$ or $ G = K_1 \vee (K_6+K_1)$.
	\end{theorem}
	
\begin{proof}
	
	For $n\geq 3$ and even $k$, let $G$ be a graph of order $n$ that satisfies the conditions of Theorem   \ref{thm7}. Assume to the contrary that $G$ is not $\text{GFC}_k$ for odd $n$ and $\text{GBC}_k$ for even $n$. By Lemma \ref{lem3} (1), there exists $\varnothing \neq S \subseteq V(G)$ such that
	\begin{equation}\label{4}
		i(G-S)\geq |S|. \tag{4}
	\end{equation}
    We now proceed to discuss the following two cases.
	
	{\textbf{Case 1.}} There are nontrivial components in $G-S$.
	
    In this case,	we can add all possible edges between $S$ and each connected component. Furthermore, each connected component, as well as  $S$, can be a clique. Let $|S|=s$ and $i(G-S)=i$, then $G\preceq K_s \vee (K_{n-i-s}+iK_1)$. Adding all edges between $V(K_{n-i-s})$ and $i-s$ vertices in $I(G-S)$,  we can transform $K_s \vee (K_{n-i-s}+iK_1)$ into a graph $K_s \vee (K_{n-2s}+sK_1)$. Through an analysis similar to the proof of Theorem \ref{thm1}, we can conclude that $e(G)\leq e(K_1 \vee (K_{n-2}+K_1))$ with equality if and only if $G=K_1 \vee (K_{n-2}+K_1)$.

	{\textbf{Case 2.}} There is no nontrivial component in $G-S$.
	
	In this case, there are  no nontrivial components in $G-S$. Let $|S|=s$, then $i(G-S)=n-s$.
	By (\ref{4})  we have $s\leq \lfloor \frac{n}{2}\rfloor$. Then $G$ is a spanning subgraph of $K_{\lfloor \frac{n}{2}\rfloor} \vee \lceil \frac{n}{2}\rceil K_1$ in this case.
	
	\vspace{\baselineskip}
	
	Now we compare the number of edges in the two cases. Following the calculation in Section 5, we have $e(K_1 \vee (K_{n-2}+K_1))> e(K_{\lfloor \frac{n}{2}\rfloor} \vee \lceil \frac{n}{2}\rceil K_1)$ for $n\geq 9$ or $n=3,7$; $e(K_1 \vee (K_{n-2}+K_1))< e(K_\frac{n}{2}\vee \frac{n}{2}K_1)$ for $n=4,6$; $e(K_1 \vee (K_{n-2}+K_1))= e(K_{\lfloor \frac{n}{2}\rfloor} \vee \lceil \frac{n}{2}\rceil K_1)$ for $n=5,8$.

In summary,	if  $G$ is not $\text{GFC}_k$ or $\text{GBC}_k$, either $G=K_1 \vee (K_{n-2}+K_1)$ or $e(G)<e(K_1 \vee (K_{n-2}+K_1))=\binom{n-1}{2}+1$ with $n\geq 9$ and $n=3,7$; if  $G$ is not $\text{GFC}_k$, then $G=K_1 \vee (K_3+K_1)$, $G=K_2 \vee 3K_1$ or $e(G)<e(K_1 \vee (K_3+K_1))=7$ with $n=5$; if  $G$ is not $\text{GBC}_k$, either $G=K_\frac{n}{2}\vee \frac{n}{2}K_1$ or $e(G)<e( K_\frac{n}{2}\vee \frac{n}{2}K_1 )=\frac{3n^2-2n}{8}$ with $n=4,6$; if  $G$ is not $\text{GBC}_k$, then $G=K_1 \vee (K_6+K_1)$, $G=K_4\vee 4K_1$ or $e(G)<22$ with $n=8$. In any case, the above leads to a contradiction with Theorem \ref{thm7}. This completes the proof.
\end{proof}

	\begin{theorem} \label{thm8}
	For $n\geq 3$ and even $k$, let $G$ be a graph of order $n$.\\
	(1)	If $n \geq 7$ or $n=3,5$, and $\rho(G)\geq \rho(K_1 \vee (K_{n-2}+K_1))$,  then $G$ is  $\text{GFC}_k$ for odd $n$ and $\text{GBC}_k$ for even $n$ unless $ G = K_1 \vee (K_{n-2}+K_1)$.\\
	(2)	If $n=4,6$ and $\rho(G)\geq \rho(K_\frac{n}{2} \vee \frac{n}{2}K_1)$, then $G$ is $\text{GBC}_k$ unless $ G = K_\frac{n}{2} \vee \frac{n}{2}K_1$.
	
\end{theorem}	
	
\begin{proof}
	
	For $n\geq 3$ and even $k$, let $G$ be a graph of order $n$ that satisfies the conditions of Theorem   \ref{thm8}. Assume to the contrary that $G$ is not $\text{GFC}_k$ for odd $n$ and $\text{GBC}_k$ for even $n$.

From Theorem  \ref{thm7}, when $n\geq 9$ or $n=3,7$, if $G\not=K_1 \vee (K_{n-2}+K_1)$, then $e(G)<\binom{n-1}{2}+1=\frac{n^2-3n+4}{2}$. Thus $e(G)\leq \frac{n^2-3n+2}{2}$. According to Lemma   \ref{lem1}, $\rho(G)\leq \sqrt{2 \frac{n^2-3n+2}{2}-n+1}\leq \sqrt{n^2-4n+3}<n-2$. However, $\rho(G)\geq \rho(K_1 \vee (K_{n-2}+K_1))\geq \rho(K_{n-1})=n-2$, which leads to a contradiction. When $n=5$, if $G\not =K_1 \vee (K_3+K_1)$ and $e(G)\geq \binom{4}{2}+1=7$, then $G=K_2\vee 3K_1$. But $\rho(G)=3<\rho(K_1 \vee (K_3+K_1))=3.09$,  which also leads to a contradiction.
	
	When $n=4,6,8$, Table \ref{tab:spectral_data} shows that the graph of order $8$ that is not $\text{GBC}_k$ with the maximum spectral radius is $K_1 \vee (K_6+K_1)$; if $n=4,6$, it is $K_\frac{n}{2}\vee \frac{n}{2}K_1$.
	 For a graph of order 8, if $\rho(G)\geq \rho(K_1 \vee (K_6+K_1))$, either $G$ is $\text{GBC}_k$ or $G =K_1 \vee (K_6+K_1)$. This contradicts our assumption. For a graph of order 4 or 6, if $\rho(G)\geq \rho(K_\frac{n}{2}\vee \frac{n}{2}K_1)$, either $G$ is $\text{GBC}_k$ or $G =K_\frac{n}{2}\vee \frac{n}{2}K_1$. This also contradicts our assumption. This completes the proof.
\end{proof}
		\vspace{\baselineskip}
	
	Observing that the following factors share the same structural characterization: $i(G-S)\leq |S|$ for all $S \subseteq V(G)$, we proceed from the definition and provide a renewed proof.

	\begin{theorem}\label{thm9}
		Let $G$ be a graph. $k\geq 2$ and $t$ are integers.
		The following conditions are equivalent.
		\\
		(1) $G$ has a $\{K_2,\{C_t\colon t\geq 3\}\}$-factor;\\
		(2) $G$ has a $\{K_2,\{C_{2t+1}\colon t\geq 1 \}\}$-factor;\\
		(3) $G$ has a fractional perfect matching;\\
		(4) $G$ has a perfect $k$-matching where $k$ is even.
	\end{theorem}
	\begin{proof}	
		
		(1)$\Rightarrow$(4): Since $G$ has a $\{K_2,\{C_t\colon t\geq 3\}\}$-factor, it has a spanning subgraph  consisting of cycles and $K_2$. We define a function $f$ that assigns a weight of $\frac{k}{2}$ to the edge in the cycles and a weight of $k$ to the $K_2$, while all remaining edges are assigned a weight of $0$. In this way, $\sum_{e\in E_G(v)} f(e)= k$ for all $v$ and we find a perfect $k$-matching of $G$.
		
		(4)$\Rightarrow$(3):
		Since $G$ has a perfect $k$-matching, there is a function $f:E(G) \rightarrow \{0,1,\ldots,k\}$ with $\sum_{e\in E(G)} f(e)= \frac{nk}{2}$ and $\sum_{e\in E_G(v)} f(e)= k$. We define a new function $f'(e)=\frac{f(e)}{k}$.  Hence, the range of $f'$ lies within the interval $[0,1]$ and  $\sum_{e\in E_G(v)} f'(e)= 1$ for all $v$. Then $G$ has a fractional perfect matching.

		(3)$\Rightarrow$(2): By Lemma    \ref{lem2}, we can define a function $f$ that assigns a value of 0, $\frac{1}{2}$ or 1 to each edge in $G$. The edges of value 1 are a matching in $G$.
		Since $G$ has a fractional perfect matching, all vertices in the graph are ``saturated". For an edge $e_0=v_0v_1$ with $f(e_0)=\frac{1}{2}$, to make $v_1$ saturate, there is another edge $e_1=v_1v_2$ incident to $v_1$, and $f(e_1)=\frac{1}{2}$. Similarly, to make $v_2$ saturate, there is an edge $e_2=v_2v_3$ incident to $v_2$\ldots  Since the graph $G$ is  finite, the vertex-edge sequence $v_0e_0v_1e_1v_2e_2v_3\ldots \ldots$will eventually return to $v_0$.
		Then edges with value $\frac{1}{2}$ will form  cycles.
		The edges of a cycle of even length can be alternately assigned the values 0 and 1, so that they become a disjoint union of copies of $K_2$. Therefore, we obtain a spanning subgraph consisting solely of $K_2$ and odd cycles. So $G$ has a $\{K_2,\{C_{2t+1}\colon t\geq 1\}\}$-factor.
		
		(2)$\Rightarrow$(1): Odd cycles factors is a special case of  cycle factors. Therefore, it is evident that (2) implies (1).
	\end{proof}
	
	From Theorem \ref{thm9} and Lemma  \ref{lem7}, we can obtain size conditions and spectral conditions for graphs that still contain a $\{K_2,\{C_t: t\geq 3\}\}$-factor after deleting any one vertex.
	
	\begin{corollary}\label{thm10}
		For $n\geq 3$, let $G$ be a graph of order $n$.\\
		(1)	If $n \geq 9$ or $n=3,7$, $e(G)\geq \binom{n-1}{2}+1$ and $ G \not= K_1 \vee (K_{n-2}+K_1)$, then $G-v$ has a $\{K_2,\{C_t: t\geq 3\}\}$-factor for any $v\in V(G)$.\\
		(2)	If $n=4,6$, $e(G)\geq \frac{3n^2-2n}{8}$ and $ G \not=K_\frac{n}{2} \vee \frac{n}{2}K_1$, then $G-v$ has a $\{K_2,\{C_t: t\geq 3\}\}$-factor for any $v\in V(G)$.\\
		(3) If $n =5$, $e(G)\geq 7$, then $G-v$ has a $\{K_2,\{C_t: t\geq 3\}\}$-factor for any $v\in V(G)$ unless $ G = K_2 \vee 3K_1$ or $ G = K_1 \vee (K_3+K_1)$.	
		\\
		(4) If $n =8$, $e(G)\geq 22$, then $G-v$ has a $\{K_2,\{C_t: t\geq 3\}\}$-factor for any $v\in V(G)$ unless $ G = K_4 \vee 4K_1$ or $ G = K_1 \vee (K_6+K_1)$.
	\end{corollary}

	\begin{corollary} \label{thm11}
		For $n\geq 3$, let $G$ be a graph of order $n$.\\
		(1)	If $n \geq 7$, $n=3,5$, $\rho(G)\geq \rho(K_1 \vee (K_{n-2}+K_1))$ and $ G \not= K_1 \vee (K_{n-2}+K_1)$, then $G-v$ has a $\{K_2,\{C_t: t\geq 3\}\}$-factor for any $v\in V(G)$.\\
		(2)	If $n=4,6$, $\rho(G)\geq \rho(K_\frac{n}{2} \vee \frac{n}{2}K_1)$ and $ G \not=K_\frac{n}{2} \vee \frac{n}{2}K_1$, then $G-v$ has a $\{K_2,\{C_t: t\geq 3\}\}$-factor for any $v\in V(G)$.
		
	\end{corollary}	
	
	\noindent \textbf{Remark.}  If we replace ``$\{K_2,\{C_t: t\geq 3\}\}$-factor"  with ``$\{K_2,\{C_{2t+1}\colon t\geq 1 \}\}$-factor" or ``fractional perfect matching" in Corollaries \ref{thm9} and \ref{thm10}, the results still hold.



\vskip 3mm
\noindent \textbf{Data availability} No data was used for the research described in the article.

\section*{Declarations}

\noindent \textbf{Conflicts of Interest}
The authors declare that there are no competing interests relevant to the content of this article.


\end{document}